\documentclass[reqno,english]{amsart}

\usepackage{amsmath,amsfonts,amssymb,graphicx,amsthm,enumerate,url}
\usepackage[noadjust]{cite}
\usepackage{stmaryrd}
\usepackage{comment,paralist}
\usepackage{mathrsfs,booktabs,tabularx}
\usepackage{xifthen,xcolor,tikz,setspace}
\usetikzlibrary{decorations.markings,patterns,shapes,calc,decorations}
\usepackage{mathtools}

 \usepackage[letterpaper]{geometry}
\usepackage[colorinlistoftodos]{todonotes}
\usepackage[colorlinks=true]{hyperref}

\numberwithin{equation}{section}

\newcommand{\Bin}{\operatorname{Bin}}
\newcommand{\Po}{\operatorname{Po}}

\renewcommand{\epsilon}{\varepsilon}

\usepackage{color}

\newtheorem{maintheorem}{Theorem}
\newtheorem{theorem}{Theorem}[section]

\newtheorem{lemma}[theorem]{Lemma}

\newtheorem{observation}[theorem]{Observation}
\newtheorem*{observation*}{Observation}

\newtheorem{corollary}[theorem]{Corollary}

\newtheorem{remark}[theorem]{Remark}

\theoremstyle{definition}{

\newtheorem*{definition*}{Definition}

}

\newcommand{\E}{\mathbb E}

\renewcommand{\P}{\mathbb P}

\newcommand{\Z}{\mathbb Z}

\newcommand{\cB}{\mathcal B}

\newcommand{\cD}{\mathcal D}
\newcommand{\cE}{\mathcal E}
\newcommand{\cF}{\mathcal F}
\newcommand{\cG}{\mathcal G}
\newcommand{\cH}{\mathcal H}

\newcommand{\cX}{\mathcal X}
\newcommand{\cY}{\mathcal Y}
\newcommand{\sL}{\mathscr L}
\newcommand{\sfC}{\mathsf C}
\newcommand{\sfM}{\mathsf M}
\newcommand{\tmf}{\mathfrak{t}}
\newcommand{\mmf}{\mathfrak{m}}

\newcommand{\llb}{\llbracket}
\newcommand{\rrb}{\rrbracket}

\newcommand{\circum}{L_{\max}}

\makeatletter
\newcommand{\tpmod}[1]{{\@displayfalse\pmod{#1}}}
\makeatother

\begin{document}

\title{Cycle lengths in sparse random graphs}

\author{Yahav Alon}
\address{Y.\ Alon\hfill\break
School of Mathematical Sciences\\ 
Raymond and Beverly Sackler Faculty of Exact Sciences\\
Tel Aviv University\\ Tel Aviv 6997801, Israel.}
\email{yahavalo@tauex.tau.ac.il}

\author{Michael Krivelevich}
\address{M.\ Krivelevich\hfill\break
School of Mathematical Sciences\\ 
Raymond and Beverly Sackler Faculty of Exact Sciences\\
Tel Aviv University\\ Tel Aviv 6997801, Israel.}
\email{krivelev@tauex.tau.ac.il}

\author{Eyal Lubetzky}
\address{E.\ Lubetzky\hfill\break
Courant Institute of Mathematical Sciences\\ New York University\\
251 Mercer Street\\ New York, NY 10012,~USA.}
\email{eyal@courant.nyu.edu}

\title{Cycle lengths in sparse random graphs}

\begin{abstract}
We study the set $\sL(G)$ of lengths of all cycles that appear in a random $d$-regular $G$ on $n$ vertices for a fixed $d\geq 3$, as well as in Erd\H{o}s--R\'enyi random graphs on $n$ vertices with a fixed average degree $c>1$. Fundamental results on the distribution of cycle counts in these models were established in the 1980's and early 1990's, with a focus on the extreme lengths: cycles of fixed length, and cycles of length linear in $n$. 
Here we derive, for a random $d$-regular graph, the limiting probability that $\sL(G)$ simultaneously contains the entire range $\{\ell,\ldots,n\}$ for $\ell\geq 3$, as an explicit expression $\theta_\ell=\theta_\ell(d)\in(0,1)$ which goes to $1$ as $\ell\to\infty$. For the random graph $\cG(n,p)$ with $p=c/n$, where $c\geq C_0$ for  some absolute constant $C_0$, we show the analogous result for the range $\{\ell,\ldots,(1-o(1))\circum(G)\}$, where $\circum$ is the length of a longest cycle in $G$. 
The limiting probability for $\cG(n,p)$ coincides with $\theta_\ell$ from the $d$-regular case when $c$ is the integer $d-1$. 
In~addition, for the directed random graph $\cD(n,p)$ we show results analogous to those on $\cG(n,p)$, and for both models we find an interval of $c \epsilon^2 n$ consecutive cycle lengths in the slightly supercritical regime $p=\frac{1+\epsilon}n$.
\end{abstract}

\maketitle

\section{Introduction} \label{sec-intro} 

We study the set $\sL(G)$ of cycle lengths appearing in a random graph $G$ on $n$ vertices with constant average degree under the classical random graph distributions: the random regular graph $\cG(n,d)$ (the uniform distribution over $d$-regular simple graphs on $n$ vertices) and the (Erd\H{o}s--R\'enyi) binomial random graph $\cG(n,p)$ (each undirected edge $ij$ for $1\leq i<j\leq n$ appears with probability $p$, independently of the other edges). We also consider $\cD(n,p)$, the directed analog of $\cG(n,p)$, which has  $n(n-1)$ i.i.d.\ Bernoulli($p$) edge variables.

Much is known about the distribution of cycles in these random graph models (see~\S\ref{sec:related-work} for a brief account), including (a) the convergence of the joint law of the variables $\{Z_k\}_{k\geq 3}$, counting the number of $k$-cycles in~$G$, to the joint law of independent Poisson random variables with explicit means; and (b) typical existence of cycles of linear length in $G\sim\cG(n,p)$ when $p=c/n$ for any $c>1$, as well as Hamilton cycles in $G\sim\cG(n,d)$.

Our results here demonstrate that the existence of cycle whose lengths are at the extreme ends of this spectrum dominate the behavior of the set $\sL(G)$ of all cycle lengths appearing in these random graphs: We~find the probability that $\sL(G)$ contains the entire range from a given fixed $\ell$ all the way to $n$ in $\cG(n,d)$, or to $(1-\epsilon)L$ where $L$ is the length of a longest cycle in $\cG(n,p)$, converges to a limit $0<\theta<1$ as $n\to\infty$.
Define the quantity $0<\theta(c,\ell)<1$ to be
\begin{equation}\label{eq:theta-def}
\theta(c,\ell) := \prod_{k=\ell}^\infty(1-e^{-c^k/(2k)})\qquad\mbox{for $c>1$ and $\ell\geq 3$}\,.
\end{equation}
(We use $\llb a,b\rrb$ to denote $\{k\in\Z: a\leq k \leq b\}$;  an event $E_n$  holds \emph{with high probability} (w.h.p.) if $\P(E_n)\to 1$.)

\begin{maintheorem} \label{thm:G(n,d)}
For every fixed $d\geq 3$, the random regular graph $G\sim\cG(n,d)$ satisfies that for every fixed $\ell\geq 3$,
\begin{equation}\label{eq:Gnd-cycles-from-constant} \lim_{n\to\infty}\P( \llb \ell ,n\rrb\subset\sL(G)) = \theta(d-1,\ell)\,.\end{equation}
In particular, $G$ contains all cycle lengths between $\ell$ and $n$ with probability at least $1-2e^{-(d-1)^\ell/(2\ell)}$.
\end{maintheorem}
Taking $\ell\to\infty$ in the above theorem shows  
$ \llb \omega_n,n\rrb\subset\sL(G) $ w.h.p.\ for every  $\omega_n$ with $\lim_{n\to\infty}\omega_n=\infty$.

\begin{maintheorem} \label{thm:G(n,p)}
There exists some $C_0>0$ so that, if $G\sim \cG(n,p)$ where $p=\frac{c}{n}$ for almost every $c>C_0$ fixed, then for every fixed $\ell\geq 3$ and any fixed $\epsilon>0$,
\begin{equation}\label{eq:Gnp-cycles-from-constant} \lim_{n\to\infty}\P( \llb \ell ,(1-\epsilon)\circum(G) \rrb\subset\sL(G)) = \theta(c,\ell)\,.\end{equation}
In addition, if $G\sim\cD(n,p)$ where $p=\frac{c}n$ with $c>C_0$ fixed, then the analog of~\eqref{eq:Gnp-cycles-from-constant} holds true 
with respect to the modified quantity $\theta'(c,\ell):=\prod_{k\geq \ell}(1-\exp(-c^k/k))$.
\end{maintheorem}

Taking $\ell\to\infty$ here yields
$ \llb \omega_n,(1-\epsilon)\circum(G)\rrb\subset\sL(G) $ w.h.p.\ for every $\omega_n$ with $\lim_{n\to\infty}\omega_n=\infty$.

\begin{remark}\label{rem:any-c>1}
Our results on $\cG(n,p),\cD(n,p)$, though presented in Theorem~\ref{thm:G(n,p)} for $p=\frac{c}n$ with a.e.\ $c>C_0$, 
address the entire supercritical regime $c>1$. In lieu of the range $\llb\ell,(1-o(1))\circum(G)\rrb$ in that theorem, one puts $\llb\ell,(1-o(1))\circum(G')\rrb$ for an analogous random graph $G'$ with edge probability $p'=(1-o(1))p$ (see Corollary~\ref{cor:G(n,p)-G'}). In particular, whenever $\circum(G)/n$ converges in probability to a left continuous limit~$f(c)$ (known to hold for a.e.\ $c>C_0$, see~\S\ref{sec:related-work}), one can further replace $\circum(G')$ by $(1-o(1))\circum(G)$, as above.
\end{remark} 

\begin{remark}\label{rem:slightly-supercritical}
Considering $\cG(n,p),\cD(n,p)$ for $p=c/n$ with $c=1+\epsilon$ for a sufficiently small $\epsilon>0$, and with the previous remark in mind, we can deduce that $G\sim\cG(n,p)$ has 
$ \P(\llb\ell,(\tfrac43-o(1))\epsilon^2 n\rrb\subset\sL(G)) \to \theta(c,\ell)$
for every fixed $\ell\geq 3$, and the same holds for $G\sim\cD(n,p)$ with the limiting constant $\theta'(c,\ell)$. 
Furthermore, one can replace $\frac43$ by any constant $\gamma_0$ that is known to lower bound $\circum(G)/(\epsilon^2 n)$ in probability; see Theorem~\ref{thm:G(n,p)-D(n,p)-slightly-supercritical}. 
\end{remark}

\subsection{Related work}\label{sec:related-work}
Following is an account, by no means exhaustive, of related results on the distribution of cycles in random graphs. For more information, the reader is referred to~\cite{Bollobas01,JLR00} and the references therein.

The random variables counting the number short cycles in $\cG(n,d)$ and $\cG(n,p)$ are well-known to be asymptotically independent Poisson. This was first established in $G\sim\cG(n,d)$ for $d\geq 3$ fixed by
Bollob\'{a}s~\cite{Bollobas80} and by Wormald~\cite{Wormald81}, where it was shown that for every integer $K$, the joint law of the random variables counting the number of $k$-cycles in $G$ for $k=3,\ldots,K$ converges weakly to the law of independent Poisson random variables $(Z_{k})_{k=3}^K$ with respective means $\lambda_k=(d-1)^k/(2k)$. The analogous statement for $\cG(n,p)$ when $p=c/n$ for fixed $c>0$ was shown by Bollob\'as in 1981 (see~\cite[\S4.1]{Bollobas01}) and independently by Karo\'nski and Ruci\'nski~\cite{KR83} with respect to (w.r.t.) $\lambda_k = c^k/(2k)$. One can immediately recognize the limiting probabilities $\theta(d-1,\ell)$ in~\eqref{eq:Gnd-cycles-from-constant} and $\theta(c,\ell)$ in~\eqref{eq:Gnp-cycles-from-constant} as the probability that $\bigcap_{k=\ell}^K \{ Z_{k} \neq 0\}$ in the respective limit as $K\to\infty$.

At the other extreme, longest  cycles in $\cG(n,d)$ and $\cG(n,p)$ were the subject of intensive study.
In $\cG(n,p)$ when $p=c/n$ for fixed $c>1$, works by Bollob\'as~\cite{Bollobas82} and by Bollob\'as, Fenner and Frieze~\cite{BFF84} culminated in the result of Frieze~\cite{Frieze86} that w.h.p.\ there exists a cycle in $G$ going through all but $(1+\epsilon_c)ce^{-c}n$ vertices, where $\epsilon_c\to 0$ as $c\to\infty$ (note that w.h.p. $G$ has $(1-o(1))ce^{-c} n$ vertices of degree 1, thus this is sharp up to~$\epsilon_c n$). A directed analog of this result in $\cD(n,p)$ was derived by the last two authors and Sudakov in~\cite{KLS13}.

For~$\cG(n,d)$, the longstanding conjecture that the graph is Hamiltonian w.h.p.\ for every fixed $d\geq 3$ was finally settled in the seminal works of Robinson and Wormald \cite{RW92,RW94}, which introduced the small subgraph conditioning  (\textsc{ssc}) method. These were followed by the paper of Janson~\cite{Janson95}, demonstrating how the \textsc{ssc} method allows one both to recover the distribution of the number of Hamilton cycles,
and, remarkably, to show contiguity of models of random regular graphs. Our analysis of $\cG(n,d)$ will rely on these results.

Letting $\circum(G)$ denote the length of a longest cycle in $G$ (its circumference), $\circum(G) / n$ is expected to converge in probability when $p=c/n$ for every fixed $c>1$, yet till recently this was not known for \emph{any}~$c>1$.
Anastos and Frieze~\cite{AF19} then proved that this holds when $c>C_0$ for some absolute constant $C_0$, and further identified the limit $f(c)$. The analogous result for $\cD(n,p)$ was thereafter obtained by the same authors in~\cite{AF20}. For $c=1+o(1)$ outside the critical window, $\circum(G)$ is known up to constant factors~\cite{Luczak91b}; see Remark~\ref{rem:constants-gamma0-gamma1}.

For $G\sim\cG(n,p)$ in the denser regime, Cooper and Frieze~\cite{CF87} proved that
 if $np - \log n - \log \log n \rightarrow \infty$ then w.h.p.\ $\sL(G) = \llb 3,n \rrb$, a property referred to as \emph{pancyclicity}.
{\L}uczak~\cite{Luczak91a} obtained that if $np \rightarrow \infty$, then for every fixed $\epsilon>0$, the graph $G$ contains all cycle lengths up to $n-(1+\epsilon)N_1$ w.h.p., where $N_1$ is the number of vertices of degree 1 in $G$. Cooper~\cite{Cooper91,Cooper92} later proved that if $np - \log n - \log \log n \rightarrow \infty$ then w.h.p.\ $G\sim \cG(n,p)$ contains a Hamilton cycle $H$ such that, for every $3\leq \ell \leq n-1$, one can construct a cycle of length $\ell$ in $G$ that contains only edges of $H$ and at most one additional edge.

Recently, Friedman and Krivelevich~\cite{FK20} studied $\sL(G)$ for certain classes of expander graphs $G$ on $n$ vertices, showing that $\sL(G)$ then contains an interval of $\delta n$ cycle lengths, for a constant $\delta>0$ that depends on the expansion parameters. Combined with well-known results on expansion in random graphs, this implies that  for every $\delta$ there exists $c_0$ such that for $c>c_0$, $d>c_0$, the graphs $G\sim\cG(n,c/n)$, $G\sim\cG(n,d)$ w.h.p.\ have that the set $\sL(G)$ of cycle lengths contains an interval of length $(1-\delta)n$ (see~\cite{FK20} for further details).

\subsection{Proof techniques}
For Theorem~\ref{thm:G(n,d)}
when $n$ is even, the aforementioned contiguity results reduce the model to the union of a Hamilton cycle and an independent uniform perfect matching. Theorem~\ref{thm:Ham-plus-PM} shows that such a random graph contains a cycle of length $\ell(n)$ with probability $1-O(\exp[-c\min(\ell,n-\ell)])$, for an absolute constant $c>0$. (A similar result holds for a union of two independent uniform Hamilton cycles, pertinent to the case of $n$ odd). The proof relies on a switching argument akin to the approach of~\cite{Luczak91a}; here, the matching edges play two roles: (I) a fraction of them, together with the Hamilton cycle, creates a large set of ($\ell-1$)-paths; (II) another fraction of those is then used to close an $\ell$-cycle. 
 Theorem~\ref{thm:G(n,p)} is proved by an analogous analysis for a union of a long cycle and  $\cG(n,\frac\delta n)$  (Theorem~\ref{thm:Ham-plus-Gnp}) or $\cD(n,\frac\delta n)$   (Theorem~\ref{thm:Ham-plus-Dnp}).

\section{random regular graphs}

Our proof will be derived from the following ingredients via contiguity properties of random regular graphs. 
The first (and main) ingredient will treat cycles in $\cG(n,3)$ whose lengths are in the range $\llb\omega_n,n-\omega_n\rrb$ for any $\omega_n\gg 1$, by means of studying the contiguous model $\cH(n)+\cG(n,1)$, the 3-regular multigraph on $n$ vertices obtained from the union of a Hamilton cycle and an independently and uniformly chosen perfect matching.
\begin{theorem} \label{thm:Ham-plus-PM}
Let $G\sim\cH(n)+\cG(n,1)$ be the random cubic $n$-vertex multigraph ($n$ even) which is the union of a Hamilton cycle and an independently chosen uniform perfect matching.
There are absolute constants $C,c>0$ so that, for any $4 \leq \ell \leq n/2$, we have $ \llb \ell, n-\ell+4 \rrb \subset \sL(G) $ with probability at least $1-C \exp(-c\ell)$. 
The same holds for $n$ odd when $G\sim \cH(n)+\cH(n)$, a union of two independent uniform Hamilton cycle.
\end{theorem}
While the above theorem shows that $\llb 4,n\rrb\subset\sL(G)$ with a probability that is uniformly bounded away from $0$, its estimate on this probability is not sharp. To obtain the correct limiting probability for this event (and more generally, for the event $\{\llb\ell,n\rrb\subset\sL(G)\}$ for any fixed $\ell$), we must treat large cycles more carefully.
Namely, the range $\llb n-\omega_n,n\rrb$ is treated by the next theorem, proved via a reduction to a result of Robinson and Wormald~\cite{RW01} on Hamilton cycles avoiding a set of random edges while including another such set.
\begin{theorem}\label{thm:G(n,d)-near-n} Let $G\sim\cG(n,d)$ for $d\geq 3$ fixed. There exists some sequence $\omega_n$ going to infinity with $n$ (sufficiently slowly) such that $\llb n-\omega_n,n\rrb\subset\sL(G)$ w.h.p.
\end{theorem}

The above two theorems will be proved in Sections~\ref{sec:G(n,d)-almost-full} and~\ref{sec:G(n,d)-near-n}, resp.
\begin{proof}[\emph{\textbf{Proof of Theorem~\ref{thm:G(n,d)}}}]
Let $\widetilde\cG(n,d)$ denote the distribution over $d$-regular multigraphs obtained via the configuration model. It is well-known (see for instance~\cite[Thm.~9.30]{JLR00}) that $\cH(n)+\cG(n,1)$ is contiguous to $\widetilde\cG'(n,3)$, the conditional distribution of $\widetilde\cG(n,3)$ given there are no loops. With probability bounded away from $0$, a graph distributed as $\widetilde\cG'(n,3)$ has no multiple edges (namely, with probability $1/e+o(1)$; see, e.g.,~\cite[Thm.~9.5]{JLR00}), and on that event it is distributed as $\cG(n,3)$. 
Similarly, when $n$ is odd, it is known that $\cH(n)+\cH(n)$ is contiguous to $\widetilde\cG'(n,4)$ (see, e.g.,~\cite[Thm.~9.41]{JLR00}), which is distributed as $\cG(n,4)$ conditional on having no multiple edges, an event whose probability is bounded away from 0 
(namely, it is $e^{-9/4}+o(1)$).

Therefore, applying Theorem~\ref{thm:Ham-plus-PM} for $\ell$ going to infinity arbitrarily slowly, and using the monotonicity of $\cG(n,d)$ in $d$ w.r.t.\ increasing properties that hold with probability $1-o(1)$ (another consequence of contiguity of random regular graphs; see~\cite[Thm.~9.36(ii)]{JLR00}) we arrive at the conclusion that $G\sim\cG(n,d)$ has
\begin{equation}\label{eq:G(n,d)-almost-full}
\P\left(\llb \omega_n,n-\omega_n\rrb\subset \sL(G)\right)\to 1\qquad \mbox{for every sequence $\omega_n$ such that $\lim_{n\to\infty}\omega_n = \infty$}\,.
\end{equation}

The treatment of the regime of $\llb \ell,\omega_n\rrb$ will follow immediately from the convergence of the short cycle distribution of $\cG(n,d)$ to asymptotically independent Poisson random variables. If $Z_{n,k}$ counts the number of $k$-cycles in $G\sim\cG(n,d)$, it is well known that for every fixed $k\geq 3$, one has $Z_{n,k} \stackrel{{\rm d}}\to \Po(\lambda_k)$ as $n\to\infty$, where $\lambda_k = (d-1)^k/(2k)$, and moreover, the joint law $\{Z_{n,k}\}_{k\geq 3}$ weakly converges as $n\to\infty$ to the joint law of independent Poisson random variables $\{Z_{\infty,k}\}_{k\geq 3}$ where $\E Z_{\infty,k} =\lambda_k$ (see, for instance,~\cite[Cor.~9.6]{JLR00})). With this in mind, fix $\epsilon>0$ and let $\omega'_n$ be the maximal integer $K\geq \ell$ satisfying
\begin{equation}\label{eq:Z-convergence} \left|\P(Z_{N,\ell}>0,\ldots,Z_{N,K}>0)-\P(Z_{\infty,\ell}>0,\ldots,Z_{\infty,K}>0)\right| < \epsilon \qquad\mbox{for all $N\geq n$}\,.\end{equation}
By the preceding discussion, $\omega'_n\to\infty$ with $n$, and so \[ \lim_{n\to\infty}\P\bigg(\bigcap_{k=\ell}^{\omega'_n} \{Z_{\infty,k}>0\}\bigg)=
 \prod_{k=\ell}^\infty (1-e^{-\lambda_k})=\theta(d-1,\ell)\,.\]
It therefore follows that
\[ \theta(d-1,\ell)-\epsilon\leq \liminf_{n\to\infty}\P\left(\llb \ell,\omega_n'\rrb\subset\sL(G)\right) \leq \limsup_{n\to\infty}\P\left(\llb \ell,\omega_n'\rrb\subset\sL(G)\right) \leq \theta(d-1,\ell)+\epsilon\,.\]
Finally, let $\omega''_n$ be the sequence with $\llb n-\omega''_n,n\rrb\subset\sL(G)$ w.h.p., specified in the conclusion of Theorem~\ref{thm:G(n,d)-near-n}.
As $\omega_n := \omega'_n \wedge \omega''_n\to\infty$ with $n$, we have $\llb \omega'_n,n-\omega''_n\rrb$ w.h.p.\  by~\eqref{eq:G(n,d)-almost-full}; letting $\epsilon\downarrow 0$ completes the proof.
\end{proof}

\subsection{Proof of Theorem~\ref{thm:Ham-plus-PM}}\label{sec:G(n,d)-almost-full}
Let $V=V(G)=\{ v_0,v_1,\ldots,v_{n-1} \}$ be such that the Hamilton cycle specified in the definition of $\cH(n)+\cG(n,1)$ is the cycle $\sfC_n = \left( v_0 ,v_1, \ldots,v_{n-1},v_0 \right)$, and let $\sfM\sim\cG(n,1)$ be the uniform independent perfect matching added to it to form the multigraph $G$.
We will show that 
\begin{equation}\label{eq:Ham-PM-single-ell-bound} \P \left( \{ \ell, n-\ell +4	 \} \not\subset \sL(G) \right) \leq C e^{-c\ell}
\quad\mbox{for every $ 4 \leq \ell \leq n/2+2 $}
 \,,\end{equation}
with $C,c>0$ being absolute constants,
from which the main result easily follows by a union bound over $\ell$.
 
 Throughout this proof, identify an undirected edge $e=\{v_i,v_j\}$ with the ordered pair $(v_i,v_j)$, where the ordering is such that $j-i\pmod n \leq n/2$. Define
\begin{equation}\label{eq:def-E_ell} E_{\ell}:=\left\{ e=\left( v_i,v_j \right) \in \tbinom{V}{2} \;:\,  j-i \tpmod n  \geq \ell/2 \right\} \end{equation} 
(so that $|E_{\ell}| = \left( \lfloor \frac{n}{2} \rfloor - \lceil \frac{\ell}{2} \rceil \right)n$),
and for every $e = (v_i,v_j)\in E_\ell$, further define the $(\lfloor \ell /2 \rfloor - 1)$-element subset
\begin{equation}\label{eq:def-F-e-ell}
F_{e,\ell } := \left\{ \{ v_{i+k\tpmod n},v_{j+\ell - k - 2 \tpmod n} \} \in \tbinom{V}{2}\,:\; 1\leq k \leq \ell /2 - 1  \right\} \,.
\end{equation}
(Note that $F_{e,\ell}$ is not assumed to be a subset of $E_\ell$, and indeed, $F_{e,\ell}\not\subset E_\ell$ e.g.\ for $\ell\sim n/2$ and $i=1$, $j\sim\ell$.)
These definitions are motivated by the next fact, owing to the classical switching principle (see Fig.~\ref{fig:switching}).
\begin{observation}[switching]\label{obs:switching}
If $e\in E_\ell$ and $f\in F_{e,\ell}$ then $\sfC_n \cup \{e,f\}$ has cycles of length $\ell$ and $n-\ell+4$.
\end{observation}
\begin{proof}
Let $e\in E_\ell$, assume w.l.o.g.\ that $e=(v_0,v_j)$ for $\ell/2 \leq j\leq  n/2 $ and let $f = \{v_k,v_{j+\ell-k-2}\}$ for some $1 \leq k \leq \ell/2-1$ (notice $j+\ell-k-2 \leq n-k < n$ by our assumption on $\ell$). The paths $P_1=(v_0,v_1,\ldots,v_k)$ and $P_2=(v_{j+\ell-k-2},\ldots,v_{j+1},v_j)$ are disjoint since $j \geq \ell/2 > k$, whence the sequence $P_1,f,P_2,e$ is an $\ell$-cycle, while $P_3,e,P4,f$ forms an $(n-\ell+4)$-cycle for $P_3=(v_{j+\ell-k-2},\ldots,v_{n-1},v_0)$ and $P_4=(v_j,\ldots,v_{k+1},v_k)$.
\end{proof}
Next, for a matching $M' \subset \sfM$, we define the auxiliary graph $\cX_{M'}$ on our original vertex set $V$ via
\[ E(\cX_{M'})= \bigcup\left\{ F_{e,\ell}\,:\; e\in M' \cap E_\ell \right\}\,.\]

\begin{observation}\label{obs:max-deg}
For any matching $M'$, the maximum degree of $\cX_{M'}$ is at most $\ell-3$.
\end{observation}
\begin{proof}
Consider $v_i\in V$ and assume w.l.o.g.\ that $i=0$. Recall its definition in~\eqref{eq:def-F-e-ell} that, for any $e\in E_\ell$, the set $F_{e,\ell}$ cannot contribute a neighbor to $v_0$ unless $e$ is incident to some vertex among $U=\{v_{n-\ell+3},\ldots,v_{n-1}\}$ (in order to get $i+k=n$ we must have that $i\in\llb n-\ell/2+1,\ldots,n-1\rrb$ whereas to get $j+\ell-k-2=n$ we must have that $j\in\llb n-\ell+3,n-\ell/2+1\rrb$). Since $M'$ is a matching, it contains at most $\ell-3$ edges $e$ incident to $U$, and each such $F_{e,\ell}$ contributes at most one neighbor to $v_0$ (being itself a matching).
\end{proof}

{\begin{figure}
    \centering
  \pgfmathsetmacro{\rad}{1.5}
    \begin{tikzpicture}[decoration={markings,
  mark=between positions 0.3 and 0.8 step 6pt
  with { \draw [fill] (0,0) circle [radius=1pt];}}]
    
    \begin{scope}[xshift=-150pt]
	\filldraw[fill=blue!10] (0,0) circle (\rad);	
	
	\draw[radius=2pt,radius=2pt,fill=gray]
    (150:\rad) circle[] node[label={[label distance=-1pt] left:{\tiny$i$}}] (ui) {}
    (60:\rad) circle[] node[label={[label distance=-5pt] above right:{\tiny$j$}}] (uj) {}
    
    (135:\rad) circle[] node[label={[label distance=-6pt, text=green!30!black] above left:{\tiny$i+1$}}] (ui1) {}
    (120:\rad) circle[] node[label={[label distance=-2pt, text=green!30!black] above:{\tiny${i+2}$}}] (ui2) {}
    (80:\rad) circle[] node[label={[label distance=-3pt, text=green!30!black] above  :{\tiny$i+\lfloor\frac{\ell}2\rfloor$-1}}] (uiK) {}
    
    (0:\rad) circle[] node[label={[label distance=-1.5pt, text=blue!30!black] right:{\tiny$j+\lceil \frac\ell2\rceil -1$}}] (ujK) {}
    (-45:\rad) circle[] node[label={[label distance=-2pt, text=blue!30!black]  right:{\tiny${j+\ell-4}$}}] (uj2) {}
    (-60:\rad) circle[] node[label={[label distance=-7.5pt, text=blue!30!black] below right :{\tiny$j+\ell-3$}}] (uj1) {};
    
    \path[purple,thick] (ui) edge [bend right=10] (uj);
	\path[blue] (ui1) edge [bend right=10] (uj1);
	\path[blue] (ui2) edge [bend right=10] (uj2);
	\path[blue] (uiK) edge [bend right=10] (ujK);
	
	\path [gray,postaction={decorate}] (ui2) arc (120:80:\rad);
    \path [gray,postaction={decorate}] (uj2) arc (-45:0:\rad);
    
	\end{scope}

	
	\filldraw[fill=blue!10] (0,0) circle (\rad);	
	
	\draw[radius=2pt,fill=none]
    (150:\rad) circle[] node[label={[label distance=-5pt] above left:{\tiny$i$}}] (vi) {}
    (60:\rad) circle[] node[label={[label distance=-5pt] above right:{\tiny$j$}}] (vj) {}
    (-30:\rad) circle[] node[label={[label distance=-5pt] below right:{\tiny${j+\ell-k-2}$}}] (vjk) {}
    (120:\rad) circle[] node[label={[label distance=-5pt] below right:{\tiny${i+k}$}}] (vik) {};
		
	\draw[red,line width=2.5pt,opacity=0.3] (vj) arc (60:-30:\rad);
	\draw[red,line width=2.5pt,opacity=0.3] (vi) arc (150:120:\rad);
	\path[red,line width=1.5pt,opacity=0.3,dashed, dash phase=3pt] (vjk) edge [bend left=30] (vik);
	\path[red,line width=1.5pt,opacity=0.3,dashed, dash phase=3pt] (vi) edge [bend left=75] (vj);
	
	\draw[blue,line width=2.5pt,opacity=0.3] (vj) arc (60:120:\rad);
	\draw[blue,line width=2.5pt,opacity=0.3] (vi) arc (150:330:\rad);
	\path[blue,line width=2pt,opacity=0.3,dashed] (vjk) edge [bend left=30] (vik);
	\path[blue,line width=2pt,opacity=0.3,dashed] (vi) edge [bend left=75] (vj);
	
	\node[circle,scale=0.4,black,fill=gray] at (vi) {};
	\node[circle,scale=0.4,black,fill=gray] at (vj) {};
	\node[circle,scale=0.4,black,fill=gray] at (vjk) {};
	\node[circle,scale=0.4,black,fill=gray] at (vik) {};
	
    \end{tikzpicture}
\vspace{-0.1in}
    \caption{The $\ell$-cycle and $(n-\ell+4)$-cycle formed by switching as per Observation~\ref{obs:switching}. On left: the edge subset set $F_{e,\ell}$ corresponding to $e=(v_i,v_j)\in E_\ell$; on right: the paths $P_1$ and $P_2$ used in the construction are in red, the paths $P_3$ and $P_4$ are in blue.}
    \label{fig:switching}
\end{figure}
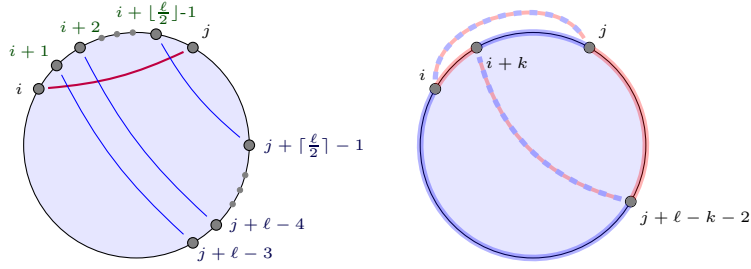}

We will expose the matching $\sfM$ in stages:
\begin{enumerate}[\!(I)]
	\item \label{it:stage-1} Expose a matching $M_1$ containing $\tmf_1:= \lfloor n/16\rfloor $ edges chosen uniformly at random out of $\sfM$.
	\item \label{it:stage-2} Expose a matching $M_2$ containing $\tmf_2 := \lfloor n/500\rfloor$ additional edges by repeatedly revealing the (random) match of a vertex with maximum degree in the subgraph induced by $\cX_{M_1}$ on the yet unmatched vertices.
	\item \label{it:stage-3} Reveal all other remaining edges of the perfect matching $\sfM$ (these will not be used by our argument).
\end{enumerate}

We will prove the following bounds for the auxiliary graph w.r.t.\ the matching $M_1$ at the end of Stage~\eqref{it:stage-1}.
\begin{lemma}\label{lem:aux-after-stage-I}
There exists $c>0$ such that, for every sufficiently large $n$, with probability at least $1-e^{-c n}$, 
the auxiliary graph $\cX_{M_1}$ has at least $ n\ell / 128$ edges, and its induced subgraph $\cY_0$ on the set $V\setminus V(M_1)$ of unmatched vertices has at least $n\ell / 200$ edges.
\end{lemma}

Modulo the above lemma, one can easily show that $M_1 \cup M_2 \cup \sfC_n$ already contains cycles of lengths $\ell$ and $n-\ell+4$ with probability $1-O(e^{-c \ell})$, implying the sought inequality~\eqref{eq:Ham-PM-single-ell-bound}. 
To see this, condition on $M_1$ and suppose that $|E(\cY_0)| \geq n\ell / 200$ as per the conclusion of the lemma. Denote by $f_1,\ldots,f_{\tmf_2}$ the edges exposed in Stage~\eqref{it:stage-2}, and let $\cY_{t}$ ($t=0,\ldots,\tmf_2-1)$ denote the induced subgraph of $\cX_{M_1}$ on the vertices yet unmatched after revealing $f_1\,\ldots,f_t$. Recall that $f_{t+1}$ will match some vertex $u$ with a maximum degree in~$\cY_{t}$. By Observation~\ref{obs:max-deg}, deleting $k$ (matched) vertices from~$\cY_0$ results in the removal of at most $k\ell$ edges; thus, for large enough $n$,
 \[|E(\cY_t)| \geq |E(\cY_0)|- 2 t\ell  \geq\frac{n\ell}{200} - 2 \tmf_2\ell \geq \frac{n\ell}{1000}\,,\]
and so $\deg(u) \geq \ell/500$ in $\cY_t$. As the match of $u$ is uniformly distributed over the other vertices of $\cY_t$, 
\[ \P\bigg( \bigcap_{t=1}^{\tmf_2} \left\{f_t \notin E(\cX_{M_1})\right\} \bigg) \leq \P\bigg( \bigcap_{t=0}^{\tmf_2-1} \left\{f_{t+1} \notin E(\cY_t)\right\} \bigg) \leq \Big(1-\frac{\ell}{500n}\Big)^{\tmf_2} \leq e^{-c \ell}\]
for an absolute constant $c>0$. The event $f_t\in E(\cX_{M_1})$ implies that $f_t\in F(e,\ell)$ for some $e\in M_1 \cap E_\ell$, in which case
Observation~\ref{obs:switching} yields the sought cycles.
 It thus remains the prove the above lemma.
 
 \begin{proof}[\textbf{\emph{Proof of Lemma~\ref{lem:aux-after-stage-I}}}]
Consider time $t=0,\ldots,\tmf_1-1$, and let $M_{1}^{(t)} = \{e_1,\ldots,e_{t}\}$ denote the first $t$ edges exposed in~$M_1$, and let $\cF_t=\sigma(M_1^{(t)})$ be the corresponding filtration. Further let
\[
S_t = \bigg\{ e\in E_{\ell} \setminus M_{1}^{(t)} \,:\; F(e,\ell )\cap \Big(\bigcup_{j\leq t}F(e_j,\ell )\Big) =\emptyset \bigg\} \,,
\]
noting that whenever $e_{t+1}\in S_t$, this edge will contribute the entire edge set $F(e_{t+1},\ell)$ as new edges to $\cX_{M_1}$.
  
Let $e=(v_i,v_j)\in E_\ell$; the edges in $F_{e,\ell}$ are a matching of consecutive pairs from $L_e=(v_{i+k\pmod n})_{k=1}^{K}$ and $R_e=(v_{j+\ell-k-2 \pmod n})_{k=1}^{K}$, with $K=\lfloor\ell/2\rfloor-1$.
So, if $f=(v_{i'},v_{j'})$ is such that $F_{e,\ell}$ and $F_{f,\ell}$ intersect, it must be that for some $d\in\llb -K+1,K-1\rrb$, either
$i' \equiv i + d $ or  $j'+\ell-K-3 \equiv i+d$ (with $\equiv$ denoting equivalence modulo $n$). For any common edge obtained as the $k$-th matched pair in $F_{e,\ell}$ and the $k'$-th pair in $F_{f,\ell}$, in the former case we would have $i'+k'\equiv i+k$ and $j'+\ell-k'-2\equiv j+\ell-k-2$, so $j'\equiv j-d$. In the latter case we have $j'+\ell-k'-2\equiv i+k$ and $i'+k'\equiv j+\ell-k-2$, and therefore $i'+j' \equiv i+j$. 
Altogether, in each case the $2(K-1)$ choices for $d$ determine the edge $f$, and hence
\begin{equation}
    \label{eq:F-e-ell-F-f-ell-intersections} \#\left\{e\in E_\ell\setminus \{e_0\} \,:\; |F_{e,\ell}\cap F_{e_0,\ell}|\neq\emptyset \right\}
\leq 4(\lfloor \tfrac\ell2\rfloor-2) \leq 2\ell -8\qquad\mbox{for every $e_0\in E_\ell$}\,.	
\end{equation}
From this bound, we immediately deduce that for every $t=0,\ldots,\tmf_1-1$,
\[ |S_t| \geq |E_\ell\setminus M_1^{(t)}| - (2\ell-8)(t-1) \geq |E_\ell| - 2\ell t\geq n(n-\ell)/2-2\ell \tmf_1 \geq (\tfrac3{16} -o(1))n^2  > n^2/6
\]
for large enough $n$; thus,
$ \P(e_{t}\in S_{t-1}\mid \cF_{t-1})\geq |S_{t-1}|/\binom{n-2t}2 > \frac13$ for all $1\leq t \leq \tmf_1$, and so the variable
\[ N_t:=\#\{1\leq t \leq \tmf_1 \,:\; e_{t} \in S_{t-1} \}\]
 stochastically dominates a $\Bin(\tmf_1,\tfrac13)$ random variable. Therefore, for some absolute constant $c>0$, 
\[   \P(N_t \leq \tfrac3{10} \tmf_1 ) \leq \exp(-c\, \tmf_1)\,.\]
As every $e_{t+1}\in S_t$ adds all of its $\lfloor\ell/2-1\rfloor$ edges to $\cX_{M_1}$,  on the event $\{N_t \geq \frac3{10}\tmf_1\}$ we have 
\[ |E(X_{M_1})|\geq (\tfrac3{20}-o(1))\ell \tmf_1 > \tfrac{1}{128} n \ell\,,\]
as claimed. To bound $|E(\cY_0)|$, define
\[ D = \sum_{t=1}^{\tmf_1} D_t\qquad\mbox{where} \qquad D_t = \# \left\{ f \in E(\cX_{M_1^{(t-1)}}) \,:\;\mbox{$e_t$ and $f$ are incident}\right\}\,,\]
i.e., $D_t$ bounds the number of edges deleted from $\cY_0$ when moving from ${M_1}^{(t-1)}$ to ${M_1}^{(t)}$ due to the edge $e_t$.
Each edge in $M_1^{(t-1)}\cap E_\ell$ adds at most $\ell/2$ edges to $\cX_{M_1^{(t)}}$, so $|E(\cX_{M_1^{(t-1)}})|\leq (t-1)\ell/2$ holds deterministically. 
The probability that a fixed edge $f\in E(\cX_{M_1^{(t-1)}})$ is incident to $e_t$ is at most $2/(n-2t)$, so
\[ \E[D_t \mid \cF_{t-1}] \leq  \frac{\ell(t-1)}{n-2t} \quad\mbox{for every $1\leq t \leq \tmf_1$}\,,\]
and in particular, $Z_t = \sum_{k=1}^t (D_k - \frac{k-1}{n-2k}\ell)$ is a supermartingale with
\[ D - Z_{\tmf_1} \leq \ell \sum_{t=1}^{\tmf_1} \frac{t-1}{n-2t} \leq \frac{\tmf_1^2}2 \frac{\ell}{n-2\tmf_1} = \frac{1+o(1)}{448}n\ell\,.\]
Recalling that $0 \leq D_t \leq 2\ell$ for every $t$ by Observation~\ref{obs:max-deg} about the maximum degree of $\cX_{M_1^{(t)}}$, whereas $-\ell \frac{t-1}{n-2t}\in [-\frac{\ell}{14},0]$, we have that $|Z_t-Z_{t-1}| \leq 2\ell$, hence Hoeffding's inequality implies that, for $\delta=10^{-4}$,
\[ \P(D > n\ell/400) \leq  \P(Z_{\tmf_1} \geq \delta n\ell)\leq \exp\bigg(-\frac{(\delta n\ell)^2}{2 (2\ell)^2 \tmf_1}\bigg) = \exp(-2\delta^2 n)\]
Overall we obtained that, for some absolute constant $c>0$,  with probability $1-O(e^{-c n})$ we have (by a union bound) both $|E(\cX_{M_1})|> n\ell/128$ and  $D < n\ell/400$, implying that $|E(\cY_0)|> n\ell/200$, as required.
\end{proof}

This completes the proof of Theorem~\ref{thm:Ham-plus-PM} for the cubic case of $\cH(n)+\cG(n,1)$. When $n$ is odd, we appeal to the same argument by treating the second independent and uniform Hamilton cycle as a matching (ordering this cycle, whenever the argument asks for the random match of a vertex $u$ we reveal its successor $v$ on the cycle, then discard $u,v$ from the pool of unmatched vertices). Note that the proof did not need the matching to be perfect, and only utilized $\lfloor n/16\rfloor$ of its edges in Step~\eqref{it:stage-1} and $\lfloor n/500\rfloor$ of its edges in Step~\eqref{it:stage-2}.
\qed

\subsection{Proof of Theorem~\ref{thm:G(n,d)-near-n}}\label{sec:G(n,d)-near-n}

The case of $\cG(n,d)$ for $d\geq 3$ reduces to the case of $\cG(n,3)$ by the monotonicity of $\cG(n,d)$ in $d$ w.r.t.\ increasing properties that hold asymptotically almost surely. Moreover, since we aim to prove that $\llb n-\omega_n,n\rrb$ for a sequence $\omega_n$ tending to $\infty$ however slowly with $n$, it suffices to show that \begin{equation}\label{eq:G(n,d)-n-k}
\P(n-k\in\sL(G))\to 1\qquad\mbox{for $G\sim\cG(n,3)$ and every fixed $k\geq 1$}\end{equation}
(where the case $k=0$---Hamiltonicity---owes of course to the famous result by Robinson and Wormald~\cite{RW92}).
The following theorem immediately implies~\eqref{eq:G(n,d)-n-k}, and will consequently establish Theorem~\ref{thm:G(n,d)-near-n}.
\begin{theorem} Let $k=k(n)$ be such that $1 \leq k = o(\sqrt n)$. Then $G\sim \cG(n,3)$ has $n-k\in\sL(G)$ w.h.p.
\end{theorem}

\begin{proof}
First consider the case where $k=2\ell$ for some $\ell\geq 1$. Let $S_G$ be a uniformly chosen set of ordered edges in $G$, chosen in the following manner. If $V(G)=\{v_1,\ldots,v_n\}$ and the $3$ half-edges of $v_i$ are denoted $(e_{i,j})_{j=1}^3$, we let $S$ be a uniform $\ell$-subset of all $e_{i,j}$'s. Clearly, these half-edges and their matches are together associated with~$2\ell$ distinct vertices except with probability $1-O(k^2/n) = 1-o(1)$ by our assumption on~$k$. Denote the vertices corresponding to the $i$-th pair of half-edges by $(u_i,u'_i)$. Further let $x_i,y_i$ and $x'_i,y'_i$ denote the other two half-edges matched in $G$ to $u_i$ and $u'_i$, respectively, once again pointing out that these half-edges are w.h.p.\ not part of $\bigcup_{i=1}^\ell \{u_i,u'_i\}$ by our assumption on $k$.

Next, define $H$ to be the graph obtained from $G$ by deleting $u_1,u'_1,\ldots,u_\ell,u'_\ell$ and thereafter connecting the half-edges $x_i y_i$ and $x'_i y'_i$ for each $i=1,\ldots,\ell$ (we add every such edge whenever the corresponding two half edges were not deleted as part of some $u_i$ or $u'_i$), denoting these newly added edges by $S_H$ (see Fig.~\ref{fig:coupling-G(n,d)-G(n-2l,d)}).

Observe that, on the event $E_1$ that the edges matched to $S_G$ in $G$ form a matching (occurring w.h.p.), the distribution of $G$ condition on these edges is uniform over perfect matchings of the remaining $3n-2\ell$ half-edges. Thereafter, on the event $E_2$ that the half-edges $\{x_i,y_i,x'_i,y'_i\}_{i=1}^\ell$ do not belong to any of the vertices $\{u_i,u'_i\}_{i=1}^\ell$ (occurring w.h.p.), these $4\ell$ half-edges are uniformly distributed over all $3(n-2\ell)$ half-edges. In conclusion, on the event $E_1\cap E_2$, which occurs w.h.p., we have that $H\sim\cG(n-2\ell,3)$, and furthermore the set $S_H$ is a uniform set of $2\ell$ ordered edges in $H$ (analogous to the above set $S_G$ in $G$).
We now appeal to a result of Robinson and Wormald~\cite[Thm.\ 3(i)]{RW01}, stating that w.h.p.\ there exists a Hamilton cycle in $H$ which \emph{avoids} a set of $2\ell=o(\sqrt n)$ randomly chosen edges $S_H$. The same cycle belongs to $G$, hence $n-2\ell\in\sL(G)$.

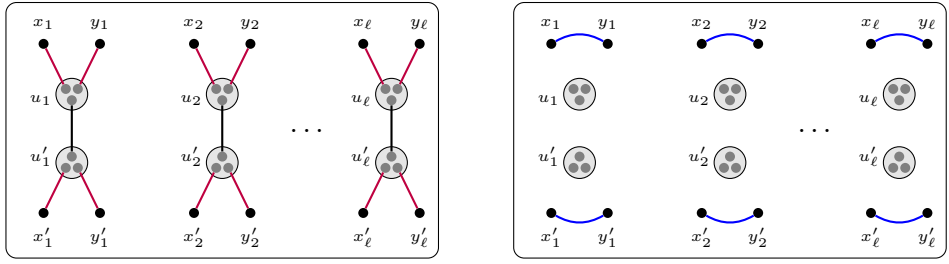
\begin{figure}
    \centering
    \begin{tikzpicture}
    \draw[rounded corners] (-.5, -2.1) rectangle (5.25, 1.3) {};
    \draw[rounded corners] (6.25, -2.1) rectangle (12, 1.3) {};
    
    \foreach \i / \b / \gr / \lbl in {1/0/g/1, 2/2/g/2, l/4.25/g/\ell, 1/6.75/h/1,2/8.75/h/2,l/11/h/\ell} {
    \filldraw[fill=gray,fill opacity=0.2]  (0.3755+\b,0.08) circle (6pt); 
    \node[circle,scale=0.4,fill=gray,label={[label distance=2pt]left:{\tiny$u_\lbl$}}] (u\gr\i0) at (0.3755+\b,0) {};
    \node[circle,scale=0.4,fill=gray] (u\gr\i1) at (0.3755+\b-0.08,0.15) {};
    \node[circle,scale=0.4,fill=gray] (u\gr\i2) at (0.3755+\b+0.08,0.15) {};
    \node[circle,scale=0.4,fill=black,label={\tiny$x_\lbl$}] (x\gr\i) at (\b,0.75) {};
	\node[circle,scale=0.4,fill=black,label={\tiny$y_\lbl$}] (y\gr\i) at (0.75+\b,0.75) {};
	
    \filldraw[fill=gray,fill opacity=0.2]  (0.3755+\b,-0.83) circle (6pt); 
    \node[circle,scale=0.4,fill=gray,label={[label distance=2pt]left:{\tiny$u'_\lbl$}}] (u\gr'\i0) at (0.3755+\b,-.75) {};
    \node[circle,scale=0.4,fill=gray] (u\gr'\i1) at (0.3755+\b-0.08,-0.9) {};
    \node[circle,scale=0.4,fill=gray] (u\gr'\i2) at (0.3755+\b+0.08,-0.9) {};
    \node[circle,scale=0.4,fill=black,label=below:{\tiny$x'_\lbl$}] (x\gr'\i) at (\b,-1.5) {};
	\node[circle,scale=0.4,fill=black,label=below:{\tiny$y'_\lbl$}] (y\gr'\i) at (0.75+\b,-1.5) {};
	}
	\node[font=\large] at (10.25,-0.4) {$\ldots$};
	\node[font=\large] at (3.5,-0.4) {$\ldots$};
	
	\foreach \i in {1,2,l} {
	\path[thick,black] (ug\i0) edge  (ug'\i0);
    \path[thick,purple] (ug\i1) edge  (xg\i);
	\path[thick,purple] (ug\i2) edge  (yg\i);
	\path[thick,purple] (ug'\i1) edge  (xg'\i);
	\path[thick,purple] (ug'\i2) edge  (yg'\i);
    \path[thick,blue] (xh\i) edge [bend left=30] (yh\i);
    \path[thick,blue] (xh'\i) edge [bend left=-30] (yh'\i);
    }
    \end{tikzpicture}
    \caption{Coupling $G\sim\cG(n,3)$ (on left) to $H\sim\cG(n-2l,3)$ (on right).}
    \label{fig:coupling-G(n,d)-G(n-2l,d)}
\end{figure}

For the case $k=2\ell-1$, we apply the same coupling and appeal to the same theorem of~\cite{RW01}, showing that w.h.p.\ there exists a Hamilton cycle in $H$ that \emph{includes} the edge $x_1 y_1$ and yet avoids the edges $S_H \setminus \{x_1 y_1\}$. This corresponds to a path on $n-2\ell$ vertices in $G$, beginning in the vertex associated to $x_1$, ending in the vertex associated with $y_1$, and avoiding $u_1$. Adding to this path the edges in $E(G)\setminus E(H)$ from $u_1$ to $x_1$ and from $u_1$ to $y_1$ closes it into a cycle of length $n-2\ell+1=n-k$; hence, again $n-k\in\sL(G)$, as required.
\end{proof}

\section{Binomial random graphs} \label{sec:G(n,p)}

In this section we derive Theorem~\ref{thm:G(n,p)}, as well as a result addressing the regime $p=\frac{1+\epsilon}n$ for small $\epsilon>0$ (Theorem~\ref{thm:G(n,p)-D(n,p)-slightly-supercritical}), as immediate consequences of results addressing $\sL(G)$ for $G$ the random graph/digraph obtained as a union of a binomial random graph and a Hamilton cycle.

Denote by $\cH(n)\oplus\cG(n,p)$ the random simple graph on the vertices $\{v_0,\ldots,v_{n-1}\}$, whose edges are  the union of the cycle $\sfC_n = \left( v_0 ,v_1, \ldots,v_{n-1},v_0 \right)$ and a random subset of all other undirected edges, each one  present independently with probability $p$. Its directed analog, denoted $\vec\cH\oplus \cD(n,p)$, has the same vertices, and its edges are the union of $\vec\sfC_n$---the directed cycle whose edges are $\{ (v_i,v_{i+1\tpmod n}):i\in\llb0,n-1\rrb\}$---and a random subset of the other (directed) edges, each one present according to an independent Bernoulli($p$) random variable.
Following are the analogs of Theorem~\ref{thm:Ham-plus-PM} for  $\cH(n)\oplus\cG(n,p)$ and $\vec\cH(n)\oplus\cD(n,p)$.

\begin{theorem}\label{thm:Ham-plus-Gnp}
Fix $\delta>0$, and let $G\sim \cH(n)\oplus\cG(n,p)$ for $p=\delta/n$. There exist absolute constants $C,c>0$ such that, for any $4\leq \ell \leq n/2$, we have $\llb \ell, n-\ell+4\rrb\subset\sL(G)$ with probability at least $1-C \exp(-c(\delta^2 \wedge 1) \ell)$.
\end{theorem}
\begin{theorem}\label{thm:Ham-plus-Dnp}
Fix $\delta>0$, and let $G\sim \vec\cH(n)\oplus\cD(n,p)$ for $p=\delta/n$. There exist absolute constants $C,c>0$ such that, for any $4\leq \ell \leq n/2$, we have $\llb \ell, n-\ell\rrb\subset\sL(G)$ with probability at least $1-C \exp(-c(\delta^2\wedge 1) \ell)$.
\end{theorem}

\subsection{Proof of Theorem~\ref{thm:Ham-plus-Gnp}}\label{subsec:pf-of-Ham-plus-Gnp}
Assume w.l.o.g.\ that $0<\delta<\frac13$.
Using the same approach as in Section~\ref{sec:G(n,d)-almost-full} (cf.\ Eq.~\eqref{eq:Ham-PM-single-ell-bound}), we will establish the theorem by showing that, for every sufficiently large $n$, 
\begin{equation}\label{eq:HAM-Gnp-single-ell-bound}
    \P \left( \{ \ell, n-\ell +4	 \} \not\subset \sL(G) \right) \leq 3 e^{-(\delta/8)^2\ell} \qquad\mbox{for every $4 \leq \ell  \leq n/2+2$}\,,
\end{equation}
 implying the statement of the lemma via a union bound.
To this end, define $E_\ell$ and $F_{e,\ell}$ for all $e\in E_\ell$ as in~\eqref{eq:def-E_ell} and~\eqref{eq:def-F-e-ell},
recalling from Observation~\ref{obs:switching} that should $G$ contain a pair of edges $e,f$ such that $e\in E_\ell$ and $f\in F_{e,\ell}$, then together with $\sfC_n$, these would give rise to cycles of lengths $\ell$ and $n-\ell+4$, as desired.

Expose the $\cG(n,p)$ part of $G$ in two stages, as $G' \cup G''$ for independent random graphs $G'\sim\cG(n,p')$ and $G''\sim\cG(n, p'')$ with $p'=p/2$ and $ p'' = p/(2-p)\geq p/2$. Letting
\[ S' = \bigcup\{ F_{e,\ell}\,:\; e\in   E(G') \cap E_\ell\}\,,\]
we will show that for every sufficiently large $n$,
\begin{equation}\label{eq:S'-size}
\P( |S'| < \tfrac{1}{30} \delta n(\ell-3) ) \leq 2\exp(-(\delta/8)^2 \ell)\,,
\end{equation}
which will establish~\eqref{eq:HAM-Gnp-single-ell-bound} and complete the proof, since on the event $|S'|\geq \frac1{30}\delta n (\ell-3)$ we will encounter a pair of edges $e,f$ with $e\in E_\ell$ and $f\in F_{e,\ell}$ via some $e\in E(G')\cap E_\ell $ and $f\in E(G'')$ except with probability
\[ \P\left(E(G'')\cap S'=\emptyset \mid G'\,,\, |S'|\geq \tfrac1{30}\delta n (\ell-3)\right) = (1-p'')^{|S'|} \leq e^{-(\delta/8)^2 \ell + \frac1{16}\delta^2 } \leq 2 e^{-(\delta/8)^2 \ell}\,.
\]
To prove~\eqref{eq:S'-size}, 
we reveal the indicators in $G'$ of potential edges from $E_\ell$ sequentially, in $\tmf:=\lceil \delta n/15\rceil$ steps. 
Step $t=1,\ldots,\tmf$ will involve revealing a sequence of indicators, until finding the first one that appears in $G'$:
\begin{enumerate}[(1)]
\item Let $A$ (resp.\ $R$) be the set of all edges of $G'$ found (resp.\ pairs in $E_\ell$ examined) in previous steps.
\item Let $\cB = \{ f\in E_\ell\setminus A \,:\; F_{f,\ell}\cap F_{e,\ell} \neq \emptyset\mbox{ for some }e\in A\}$. 
\item Order $\cE = E_{\ell} \setminus (\cB \cup R)$ in an arbitrary way, and reveal its indicators one by one:
\begin{enumerate}[(a)]
\item if a pair $f\in\cE$ corresponds to an edge of $G'$, let $A\mapsto A\cup\{f\}$ and $R\mapsto R\cup\{f\}$, then end step $t$.
\item if a pair $f\in \cE$ does not belong to $E(G')$, let $R\mapsto R\cup\{f\}$. If this results in $|R| > \mmf := \lfloor n^2 / 5\rfloor$, abort the entire process, marking it a failure. Otherwise, move on to examine the next edge in $\cE$.
\end{enumerate}
\end{enumerate}
Recalling~\eqref{eq:F-e-ell-F-f-ell-intersections}, in each step $t$ we have 
\[ |\cB| \leq (2\ell-8)|A| \leq (2\ell-8) (\tmf-1) \leq \delta n^2 / 15 < n^2 / 45\]
(using that $\delta<\frac13$). By construction, $|R_{t-1}|\leq \mmf = \lfloor n^2/5\rfloor$, whereas $|E_\ell| = \frac12 n(n-\ell) \geq n^2/4-n$, and so 
\[ |\cE| \geq |E_\ell| - \frac{n^2}{45} - \frac{n^2}5 \geq \frac{1-o(1)}{36} n^2\,, 
\]
and in particular $\cE\neq\emptyset$ for large enough $n$. 
So, the only way the process could fail is if we had $|R| > \mmf$. The latter event, in turn, occurs if and only if fewer than $\tmf$ edges were found in the first $\mmf$ exposed pairs. Thus,
\[ \P(|R| > \mmf) \leq  \P(\Bin(\mmf,p') < \tmf) \leq \exp(-\tfrac1{180} \delta n ) \leq \exp(-(\delta/8)^2 n )\,,\]
using $\P(X - \mu < -a) \leq \exp(-\frac12 a^2/\mu)$ for a binomial random variable with mean $\mu$ (e.g.,~\cite[\S2, Eq.~(2.6)]{JLR00}). 
Hence, with probability at least $1-\exp(-(\delta/8)^2 n)$, we arrive at a set $A\subset E_\ell\cap E(G')$ where the corresponding sets $\{ F_{e,\ell} \}_{e\in A}$ are pairwise disjoint by construction, thus
$ |S'|\geq (\lfloor \ell/2\rfloor - 1)\tmf  \geq \frac1{30}\delta n (\ell - 3)$, yielding~\eqref{eq:S'-size}.
\qed

\subsection{Proof of Theorem~\ref{thm:Ham-plus-Dnp}}\label{subsec:pf-of-Ham-plus-Dnp}
Assume w.l.o.g.\ that $0<\delta<\frac13$. For $4 \leq \ell \leq n-4$, define:
\begin{equation}\label{eq:def-dir-E_ell} E_{\ell}:=\left\{ \vec e=\left( v_i,v_j \right) \in V\times V \;:\,   j-i \tpmod n \in \llb 2, n-\ell\rrb  \right\} \end{equation} 
and for every $\vec e = (v_i,v_j)\in E_\ell$, let 
\begin{equation}\label{eq:def-dir-F-e-ell}
F_{\vec e,\ell } := \left\{ ( v_{j+\ell - k-2 \tpmod n}, v_{i-k \tpmod n} ) \in E_\ell\,:\;  k \in\llb0, \ell -2\rrb  \right\} \,.
\end{equation}
(To see that $F_{\vec e,\ell}\subset E_{\ell}$, w.l.o.g.\ let $\vec e=(v_0,v_j)\in E_{\ell}$ for $j \in\llb2, n-\ell\rrb$, whereby every $\vec f=(v_{i'},v_{j'})\in F_{\vec e,\ell}$ has $j'-i'=n-\ell+2-j\in\llb2,n-\ell\rrb$.) In lieu of Observation~\ref{obs:switching}, we use the following simple fact (see Fig.~\ref{fig:shortcut}).

\begin{observation}\label{obs:dir-shortcut}
If $\vec e\in E_\ell$ and $\vec f\in F_{\vec e,\ell}$ then $\vec\sfC_n \cup \{\vec e,\vec f\}$ has a directed cycle of length $\ell$.
\end{observation}
\begin{proof}
Let $\vec e \in E_{\ell}$, assuming w.l.o.g.\ that $\vec e = (v_0,v_j)$ for $j\in\llb 2,n-\ell \rrb$, and let $\vec f \in F_{\vec e,\ell}$, denoting by $k$ the index corresponding to $\vec f$ in this edge set as per Eq.~\eqref{eq:def-dir-F-e-ell}.
Since $\vec e\in E_{\ell}$, the (possibly trivial) paths $P_1=(v_j,v_{j+1},\ldots,v_{j+\ell-k-2})$, $P_2=(v_{n-k},\ldots,v_{n-1},v_0)$ are disjoint, so $(\vec e,P_1,\vec f,P_2)$ is an $\ell$-cycle. \end{proof}

{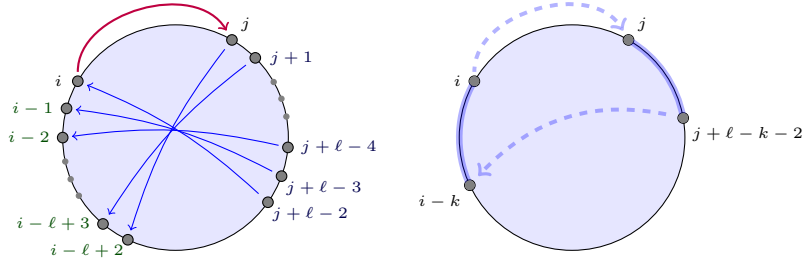
\begin{figure}
    \centering
  \pgfmathsetmacro{\rad}{1.5}
    \begin{tikzpicture}[decoration={markings,
  mark=between positions 0.35 and 0.9 step 7pt
  with { \draw [fill] (0,0) circle [radius=1pt];}}]

\begin{scope}[xshift=-150pt]
	\filldraw[fill=blue!10] (0,0) circle (\rad);	
	
	\draw[radius=2pt,radius=2pt,fill=gray]
    (150:\rad) circle[] node[label={[label distance=-1pt] left:{\tiny$i$}}] (ui) {}
    (60:\rad) circle[] node[label={[label distance=-5pt] above right:{\tiny$j$}}] (uj) {}
    
    (165:\rad) circle[] node[label={[label distance=-2pt, text=green!30!black]  left:{\tiny$i-1$}}] (ui1) {}
    (180:\rad) circle[] node[label={[label distance=-2pt, text=green!30!black] left:{\tiny${i-2}$}}] (ui2) {}
    (230:\rad) circle[] node[label={[label distance=-2pt, text=green!30!black] left  :{\tiny$i-\ell+3$}}] (uiKm1) {}
    (245:\rad) circle[] node[label={[label distance=-7.5pt, text=green!30!black] below left  :{\tiny$i-\ell+2$}}] (uiK) {}
    
    (45:\rad) circle[] node[label={[label distance=-1.5pt, text=blue!30!black] right:{\tiny$j+1$}}] (ujKm1) {}
    (-5:\rad) circle[] node[label={[label distance=-2pt, text=blue!30!black]  right:{\tiny${j+\ell-4}$}}] (uj2) {}
    (-20:\rad) circle[] node[label={[label distance=-7.5pt, text=blue!30!black] below right :{\tiny$j+\ell-3$}}] (uj1) {}
    (-35:\rad) circle[] node[label={[label distance=-7.5pt, text=blue!30!black] below right :{\tiny$j+\ell-2$}}] (uj0) {};
    
    \path[purple,thick,->] (ui) edge [bend left=75] (uj);
	\path[blue,->] (uj0) edge [bend right=10] (ui);
	\path[blue,->] (uj1) edge [bend right=10] (ui1);
	\path[blue,->] (uj2) edge [bend right=10] (ui2);
	\path[blue,->] (ujKm1) edge [bend right=10] (uiKm1);
	\path[blue,->] (uj) edge [bend right=10] (uiK);
	
	\path [gray,postaction={decorate}] (ui2) arc (180:215:\rad);
    \path [gray,postaction={decorate}] (uj2) arc (-5:40:\rad);
    
	\end{scope}
    
	\filldraw[fill=blue!10] (0,0) circle (\rad);	
	
	\draw[radius=2pt,fill=none]
    (150:\rad) circle[] node[label={[label distance=-2pt]  left:{\tiny$i$}}] (vi) {}
    (60:\rad) circle[] node[label={[label distance=-5pt] above right:{\tiny$j$}}] (vj) {}
    (10:\rad) circle[] node[label={[label distance=-5pt] below right:{\tiny${j+\ell-k-2}$}}] (vjk) {}
    (205:\rad) circle[] node[label={[label distance=-5pt] below left:{\tiny${i-k}$}}] (vik) {};
		
	\draw[blue,line width=2.5pt,opacity=0.3] (vj) arc (60:10:\rad);
	\draw[blue,line width=2.5pt,opacity=0.3] (vi) arc (150:205:\rad);
	\path[blue,line width=1.5pt,opacity=0.3,dashed, dash phase=3pt,->] (vjk) edge [bend left=-30] (vik);
	\path[blue,line width=1.5pt,opacity=0.3,dashed, dash phase=1pt,->] (vi) edge [bend left=75] (vj);
	
	\node[circle,scale=0.4,black,fill=gray] at (vi) {};
	\node[circle,scale=0.4,black,fill=gray] at (vj) {};
	\node[circle,scale=0.4,black,fill=gray] at (vjk) {};
	\node[circle,scale=0.4,black,fill=gray] at (vik) {};
	
    \end{tikzpicture}
\vspace{-0.1in}
    \caption{The edge subset set $F_{\vec e,\ell}$ corresponding to $ \vec e=(v_i,v_j)\in E_\ell$, and the $\ell$-cycle specified in  Observation~\ref{obs:dir-shortcut} using $\vec e$ and $\vec f\in F_{\vec e,\ell}$.}
    \label{fig:shortcut}
\end{figure}}

The bound on pairwise intersections of the sets $F_{\vec e,\ell}$ needed for the proofs, mirroring~\eqref{eq:F-e-ell-F-f-ell-intersections}, becomes
\begin{equation}
    \label{eq:dir-F-e-ell-F-f-ell-intersections} \#\left\{\vec e\in E_\ell \setminus \{\vec e_0\} \,:\; F_{\vec e,\ell}\cap F_{\vec e_0,\ell}\neq\emptyset \right\}
\leq 2\ell-6\qquad\mbox{for every $\vec e_0\in E_{\vec \ell}$}\,	
\end{equation}
which follows immediately from the fact that if $\vec e_1 = (v_i,v_j)\in E_\ell$ and $\vec e_2=(v_{i'},v_{j'})\in E_\ell$ are distinct edges satisfying $F_{\vec e_1,\ell} \cap F_{\vec e_2,\ell} \neq \emptyset$, then it must be the case that $i-i' \equiv d$ and $j-j'\equiv d$ for some $d\in \llb -\ell+3,\ell-3\rrb\setminus\{0\}$. 

From this point, we proceed with the procedure described in the proof of Theorem~\ref{thm:Ham-plus-Gnp}, with parameters~$\tmf$ (number of steps) and $\mmf$ (limit on the number of edges that may be exposed) given by
	\[ \tmf = \lceil \tfrac14 \delta (n-\ell-1) \rceil \qquad,\qquad  \mmf = \lfloor \tfrac34 n (n-\ell-1) \rfloor\,.\]
In every step, the set $\cB$ of edges we wish to avoid satisfies
\[ |\cB| \leq (2\ell-6)|A| \leq (2\ell-6)(\tmf-1)\leq \delta (n-\ell-1)\ell /2\,.\]
By definition $|R_{t-1}|\leq \mmf$ and $|E_\ell| = n(n-\ell-1)$, so, recalling that $\delta<\frac13$, plugging in $\ell\leq n$ yields
\[ |\cE|\geq |E_\ell|- |\cB| - \mmf 
\geq n(n-\ell-1)(1-\tfrac34 - \tfrac\delta2) \geq \tfrac1{12} n(n-\ell-1) > 0\,. \]
Now we have $ \frac38 \delta(n-\ell-2) \leq \mmf p' \leq \frac38\delta (n-\ell-1) $ and $\mmf p'-\tmf \geq \tfrac18\delta (n-\ell-1) - 2$, so (again using $\delta<\frac13$)
\[ \P(|R|>\mmf) \leq \P(\Bin(\mmf,p')<\tmf) \leq 
2\exp\left(-\tfrac1{48}\delta (n-\ell-1)\right) \leq 
2\exp\left(-\tfrac1{16}\delta^2 (n-\ell-1)\right)\,.
\]
Since each edge $\vec e \in A$ contributes $|F_{\vec e,\ell}| = \ell-1 $ unique edges to $S' = \bigcup\{ F_{\vec e,\ell} : \vec e\in A\}$, we have that $|S'| \geq \tfrac14\delta(\ell-1)(n-\ell-1)$ on the event that the above procedure was successful, whence
\[\P(E(G'') \cap S' = \emptyset \mid G'\,,\,|A|\geq \tmf) \leq (1-p'')^{|S'|} \leq \exp[ -\tfrac18 \delta^2 (\ell-1)(n-\ell-1)/n ]\,.\]
For $\ell \leq n/2$ this is at most $\exp[-\frac1{16} \delta^2(\ell-1)]$, whereas for $\ell > n/2$ this is at most $\exp[-\frac1{16} \delta^2 (n-\ell-1)]$. Altogether, we conclude that 
\[
    \P \left( \ell\notin \sL(G) \right) \leq 3 \exp\left[-(\delta/4)^2 \left((\ell-1)\wedge (n-\ell-1)\right)\right] \qquad\mbox{for every $4 \leq \ell  \leq n-4$}\,,
\]
completing the proof.
\qed

\subsection{Consequences for 
\texorpdfstring{$\cG(n,p)$}{G(n,p)} and \texorpdfstring{$\cD(n,p)$}{D(n,p)}}
From Theorems~\ref{thm:Ham-plus-Gnp} and~\ref{thm:Ham-plus-Dnp} we can readily infer the following.
\begin{corollary}\label{cor:G(n,p)-G'}
Fix $c>c'>1$ and $\gamma>0$, and let $G\sim\cG(n,p=\frac{c}n)$ and $G'\sim\cG(n,p'=\frac{c'}n)$. If $\omega_n$ and $L'_n$ are sequences such that $\omega_n\to\infty$ with $n$ and $\circum(G')\geq L'_n \geq \gamma n$ w.h.p., then $\llb \omega_n, L'_n-\omega_n\rrb\subset\sL(G)$ w.h.p.
The same conclusion holds when $G\sim\cD(n,p=\frac{c}n)$ and $G'\sim\cD(n,p'=\frac{c'}n)$ for $c,c',\omega_n,L'_n$ as above.
\end{corollary}
\begin{proof}
Via standard sprinkling, draw $G\sim\cG(n,p)$ by exposing $G'\sim\cG(n,p')$, and then for each of the missing edges, independently adding it with probability $p'' := \frac{p-p'}{1-p'}$. Reveal $G'$, and suppose that it contains a cycle $\sfC_{\ell'_n}$ of length~$\ell'_n \geq L'_n\geq \gamma n$ (an event that occurs w.h.p.\ by our hypothesis). The induced subgraph of $G$ on~$\sfC_{\ell'_n}$ dominates a copy of $\cH(\ell'_n)\oplus\cG(\ell'_n, p'')$, and $ p''\geq\frac{c-c'}n\geq \frac{\gamma (c-c')}{\ell'_n}$, so Theorem~\ref{thm:Ham-plus-Gnp} (with $\delta=\gamma(c-c')>0$) implies $\llb \omega_n,\ell'_n-\omega_n\rrb\subset \sL(G)$ (hence also $\llb\omega_n,L'_n-\omega_n\rrb\subset\sL(G)$) w.h.p. 
For $G\sim\cD(n,p)$, we use the same coupling of $G$ to $G'\cup G''$ for $G'\sim\cD(n,p')$ and $G''\sim\cD(L'_n,p'')$, whence Theorem~\ref{thm:Ham-plus-Dnp} completes the proof.
\end{proof}

\begin{proof}[\textbf{\emph{Proof of Theorem~\ref{thm:G(n,p)}}}]
Beginning with $G\sim\cG(n,p)$, we appeal to the recent result of Anastos and Frieze~\cite{AF19} (see Theorem~1.3(a) in that work) that, for some absolute $C_0>0$, if $G\sim\cG(n,p)$ with $p=c/n$ for fixed $c>C_0$ then  $\circum(G)/n \to f(c)$ in probability for some function $f(c)$. As $\{\circum(G)\geq k\}$ is a monotone increasing property, the limit $f(c)$ is necessarily monotone non-decreasing in $c$, and as such has countably many discontinuity points. Restricting our attention to every continuity point $c$, for every $\epsilon>0$ there exists $\delta>0$ such that at  $p'=(1-\delta)p$ we have $\circum(G')\geq (1-\epsilon)\circum(G)$, whence Corollary~\ref{cor:G(n,p)-G'} implies that
\begin{equation}\label{eq:G(n,p)-almost-full}
\P\left(\llb \omega_n,(1-\epsilon)\circum(G)\rrb\right)\to 1\qquad \mbox{for every sequence $\omega_n$ such that $\lim_{n\to\infty}\omega_n = \infty$}\,.
\end{equation}
Let $Z_{n,k}$ $(k\geq 3)$ be the number of $k$-cycles in $G\sim\cG(n,p=\frac{c}n)$. It is well-known (see~\cite[Cor.~4.9]{Bollobas01}) that, for every fixed integer $K$, the joint law of the variables $\{Z_{n,k}\}_{k=3}^K$ converges to that of independent Poisson random variables $\{Z_{\infty,k}\}_{k=3}^K$ where $\E Z_{\infty,k}=\lambda_k$ for $\lambda_k := c^k/(2k)$. Analogously to~\eqref{eq:Z-convergence}, fix $\epsilon'>0$ and let $\omega'_n$ be the maximal $K$ such that 
\[ \Big|\P\big(\bigcup_{k=\ell}^K \{Z_{n,k}>0\}\big) - 
 \P\big(\bigcup_{k=\ell}^K \{Z_{\infty,k}>0\}\big)\Big| < \epsilon'\qquad\mbox{for all $N\geq n$}\,.\]
The aforementioned convergence result implies that $\omega'_n\to\infty$ with $n$, so 
\[\P(\lim_{n\to\infty}\{Z_{\infty,k}>0\}) = \prod_{k=\ell}^\infty (1-e^{-\lambda_k}) =\theta(c,\ell)
\,.\] Combining this with~\eqref{eq:G(n,d)-almost-full} shows that $\P(\llb\ell,(1-\epsilon)\circum(G)\rrb\subset\sL(G))$ is within $\epsilon'+o(1)$ of $\theta(c,\ell)$, as required. 

The analogous statement for $G\sim\cD(n,p)$  follows from Corollary~\ref{cor:G(n,p)-G'} in exactly the same manner as argued above, except that now, rather than relying on~\cite{AF19}, we appeal to the sequel by the same authors~\cite{AF20} for the fact that there exists some absolute $C_0>0$ such that, if $G\sim\cD(n,p)$ with $p=c/n$ for fixed $c>C_0$ then $\circum(G)/n\to f(c)$ in probability for some (non-decreasing) function $f(c)$. Finally, the joint law of short cycles is again that of asymptotically independent Poisson random variables (e.g., via the same method-of-moments argument referenced above, and stated for arbitrary strictly balanced graphs in~\cite[Thm.~4.8]{Bollobas01}), yet now the automorphism group of a $k$-cycle in the directed graph $G$ has order $k$ rather than $2k$.
\end{proof}

\begin{remark}
The weaker statement where the absolute constant $C_0>0$ from Theorems~\ref{thm:G(n,p)} is replaced by $C_\epsilon$ may be derived from Corollary~\ref{cor:G(n,p)-G'} using much earlier works.
Namely, consider the statement  that for every $\epsilon>0$ there exists some $C_\epsilon$ so that, if $\omega_n$ is any sequence going to $\infty$ with $n$, then
\[ \P\left(\llb\omega_n,n-(1+\epsilon)c e^{-c}n \rrb\subset \sL(G)\right) \to 1\qquad \mbox{if $G\sim \cG(n,p)$ with $p=\frac c n$ for $c\geq C_\epsilon$ fixed}\,.\]
(The number of degree 1 vertices in $G\sim\cG(n,p)$---which are not part of any cycle---is typically $(ce^{-c}+o(1))n$.) 
This follows from combining Corollary~\ref{cor:G(n,p)-G'} with the result of Frieze~\cite{Frieze86} that $\circum(G)\geq (1 - (1+\epsilon_c)c e^{-c} )n$ w.h.p.\ for some sequence $\epsilon_c$ going to $0$ as $c\to\infty$.
Similarly, one obtains the analogous statement for~$\cD(n,p)$ (where there are typically $(2e^{-c}+o(1))$ vertices of 0 out-degree or 0 in-degree), namely that
\[ \P\left(\llb\omega_n,n-(2+\epsilon) e^{-c}n \rrb\subset \sL(G)\right) \to 1\qquad \mbox{if $G\sim \cD(n,p)$ with $p=\frac c n$ for $c\geq C_\epsilon$ fixed}\,,\]
via Corollary~\ref{cor:G(n,p)-G'} and the $\cD(n,p)$ analog of said result of~\cite{Frieze86}, due to the last two authors and Sudakov~\cite{KLS13}.
\end{remark}

We next address the setting of $\cG(n,p)$ and $\cD(n,p)$ when $p=\frac{1+\epsilon}n$ for a small $\epsilon>0$.
{\L}uczak~\cite{Luczak91b} established the existence of constants $0<\gamma_0<\gamma_1$ and $\epsilon_0>0$ such that
\begin{equation}\label{eq:circum-supercritical-G(n,p)} \P(\circum(G)\in \llb \gamma_0\epsilon^2 n,\gamma_1\epsilon^2 n\rrb)\to 1 \quad\mbox{if $G\sim\cG(n,p)$ with $p=\frac{1+\epsilon}n$ for $0<\epsilon<\epsilon_0$ fixed}\,.\end{equation}
\begin{remark}\label{rem:constants-gamma0-gamma1}
This statement was proved in~\cite{Luczak91b} for the constants $\gamma_0=\frac43$ and $\gamma_1 
= \frac43(1 + \log\frac32 ) < 1.874$
for the slightly supercritical case $\epsilon=o(1)$, $\epsilon^3 n\to\infty$. These constants are easily explained via the description of the giant component of $G$ as having a \emph{kernel} $\mathcal{K}\sim\cG(N,3)$ with $N\sim\frac43 \epsilon^3 n$ vertices (and $\sim2\epsilon^3 n$ edges), inflated into a 2-core by replacing every edge by a path of length i.i.d.\  Geometric($\frac{1}\epsilon$) (see~\cite{DKLP11} for a formal statement of this description).
In this case, the kernel is Hamiltonian w.h.p.\ by~\cite{RW92}, giving the constant~$\gamma_0$. A cycle may visit at most two of the edges incident to any vertex in the kernel, thus taking the longest $\frac23$ of the paths replacing the edges of $\mathcal{K}$, combined with the classical representation of order statistics for i.i.d.\ exponential variables, yields the constant $\gamma_1$. For improved constants replacing $\gamma_0$ and $\gamma_1$, see, e.g.,~\cite{KW13}.
The analogous description of the strictly supercritical giant component~\cite{DLP14} extends~\eqref{eq:circum-supercritical-G(n,p)} to $0<\epsilon<\epsilon_0$ fixed.
\end{remark}
The elegant coupling argument of McDiarmid~\cite{McDiarmid80} immediately extends~\eqref{eq:circum-supercritical-G(n,p)} to $G\sim\cD(n,p)$.

\begin{theorem}\label{thm:G(n,p)-D(n,p)-slightly-supercritical}
Suppose that there exist absolute constants $\gamma_0,\epsilon_0>0$ such that
\[ \P(\circum(G) \geq \gamma_0 \epsilon^2 n) \to 1\qquad\qquad\mbox{if $G\sim\cG(n,p)$ for $p=\frac{1+\epsilon}n$ for fixed $0<\epsilon<\epsilon_0$}\,.\]
Then for every $0<\epsilon<\epsilon_0$ and every fixed $\ell\geq 3$, the random graph $G\sim\cG(n,p)$ with $p=(1+\epsilon)/n$ has 
\[ \P(\llb \ell,(1-\delta)\gamma_0 \epsilon^2 n\rrb\subset\sL(G))\to \theta(c,\ell)\qquad\mbox{for every fixed $\delta>0$}\,.\]
The same statement holds for $G\sim\cD(n,p)$ with $p=\frac{1+\epsilon}n$ when replacing $\theta(c,\ell)$ by $\theta'(c,\ell)$ as in Theorem~\ref{thm:G(n,p)}.
\end{theorem}
\begin{proof}
Let $G\sim\cG(n,p)$ (the same argument will cover $G\sim\cD(n,p)$, as Corollary~\ref{cor:G(n,p)-G'} holds for both models). Fix $0<\epsilon<\epsilon_0$, let $0<\delta<1$, and define $c'=1+\epsilon\sqrt{1-\delta}$. Then w.h.p.\  $G'\sim\cG(n,p')$ with $p=c'/n$ has $\circum(G') \geq (1-\delta)\gamma_0 \epsilon^2 n$ by assumption, so $\llb \omega_n,(1-\delta)\gamma_0\epsilon^2 n-\omega_n \rrb\subset \sL(G)$ w.h.p.\ by Corollary~\ref{cor:G(n,p)-G'}.
The range $\llb \ell,\omega_n\rrb $ is covered exactly as in the proof of Theorem~\ref{thm:G(n,p)}, and produces the limiting probability $\theta(c,\ell)$ (and its modified version $\theta'(c,\ell)$ in the directed case $G\sim\cD(n,p)$).
\end{proof}

\begin{remark} The machinery developed in Theorem~\ref{thm:Ham-plus-Gnp} can be applied also to models of random graphs set  by adding random edges to base graphs with a given property. For example, it implies through the results of~\cite{KRS15} that adding $\delta n$ random edges to a tree $T$ on $n$ vertices with maximum degree bounded by $\Delta$ produces typically a graph G with the set $\sL(G)$ containing all cycle lengths in $\llb\omega_n,cn\rrb$ for $c=c(\delta,\Delta)>0$. The proof proceeds using the first portion of random edges to create w.h.p.\ a linearly long cycle $C$ (see~\cite[Thm.~6]{KRS15}), and then by applying Theorem~\ref{thm:Ham-plus-Gnp} to $C$ with the second portion of random edges.
\end{remark}
\subsection*{Acknowledgment} M.K.~was supported in part by
 USA-Israel BSF grant 2018267 and ISF grant 1261/17.
E.L.~was supported in part by NSF grant DMS-1812095.

\bibliographystyle{abbrv}
\bibliography{refs}

\end{document}